\documentclass[11pt]{article}
\usepackage{fancyhdr}
\usepackage{amsthm}
\usepackage{amsmath}
\usepackage{amssymb}
\usepackage{amstext}
\usepackage{graphicx} 
\usepackage{psfrag}
\usepackage{cite}
\usepackage{pstricks}

 \definecolor{lightgray}{gray}{.75}
\usepackage{booktabs}

\usepackage[small,nohug,heads=vee]{diagrams}
\diagramstyle[labelstyle=\scriptstyle]

\usepackage{float}
\usepackage{graphics,graphicx} %% comment for "draft" mode w/o pictures
\usepackage[backref,colorlinks=true,linkcolor=blue,citecolor=red,urlcolor=blue,citebordercolor={0 0 1},urlbordercolor={0 0 1},linkbordercolor={0 0 1}]{hyperref} %needs to be loaded after most things

\newtheorem{thm}{Theorem}[section]
\newtheorem{theorem}[thm]{Theorem}
\newtheorem*{theorem*}{Theorem}
\newtheorem{definition}[thm]{Definition}
\newtheorem{proposition}[thm]{Proposition}
\newtheorem{lemma}[thm]{Lemma}
\newtheorem{claim}[thm]{Claim}
\newtheorem{corollary}[thm]{Corollary}

\newtheorem{remark}[thm]{Remark}
\numberwithin{equation}{section}
\newcommand{\CP}{\mathbb{CP}}
\newcommand{\D}{\mathbb{D}}
\newcommand{\CX}{\mathbb{C}}
\newcommand{\Z}{\mathbb{Z}}
\newcommand{\R}{\mathbb{R}}

\newcommand{\A}{\mathcal{A}}

\newcommand{\Lag}{\mathcal{L}}

\newcommand{\bu}{\textbf{u}}
\newcommand{\bv}{\textbf{v}}
\newcommand{\bw}{\textbf{w}}

\newcommand{\uu}{u_+^\infty}
\newcommand{\del}{\partial}
\newcommand{\kl}{(k^2,kl - 1)}
\newcommand{\abc}{(a^2,b^2,c^2)}

\begin{document}
\thispagestyle{empty}
\begin{LARGE}
\begin{center}
\textbf{Infinitely many exotic monotone Lagrangian tori in $\CP ^2$}
\end{center}

 \end{LARGE}

\begin{large}

\begin{center}
\textbf{Renato Ferreira de Velloso VIANNA} 
\end{center}
\end{large}

\begin{center}
  
\abstract  Related to each degeneration from $\CP^2$ to
$\CP(a^2,b^2,c^2)$, for $(a,b,c)$ a Markov triple - see \eqref{Markov} - there
is a monotone Lagrangian torus, which we call $T(a^2,b^2,c^2)$. We employ
techniques from symplectic field theory to prove that no two of them are
Hamiltonian isotopic to each other.
 
 \end{center}

\tableofcontents

\newpage

\section{Introduction}
In \cite{RV}, we explicitly constructed a monotone
Lagrangian torus in $\CP^2$, which we named $T(1,4,25)$. Moreover, we computed
the number of Maslov index 2 discs bounded by $T(1,4,25)$, to prove it is not
Hamiltonian isotopic to the known Clifford and Chekanov tori.

An almost toric fibration is a singular Lagrangian torus fibration allowing
nodal (pinched torus) and elliptic (circles or points) singularities; see
Definition 2.9 of \cite{RV}. The $T(1,4,25)$ Lagrangian torus can be seen as the
`central' fiber of a particular almost toric fibration of $\CP^2$. This almost
toric fibration can be obtained from the standard toric fibration of $\CP^2$ by
a series of operations called \emph{nodal trades} and \emph{nodal slides} that
don't change the symplectic four manifold - see Definitions 2.12, 2.13 of
\cite{RV}. Nodal trade replaces a corner (corank 2 elliptic singularity) by a
nodal fiber in the interior of the fibration with a cut that encodes the
monodromy around the nodal fiber. Nodal slides amount to lengthening and
shortening the cut. The base diagram for the almost toric fibration containing
the $T(1,4,25)$ monotone Lagrangian torus can be arranged to look similar to the
base for the standard toric fibration of the orbifold weighted projective space
$\CP(1,4,25)$, but with nodal fibers and cuts replacing the orbifold points -
see Figure \ref{figIntro}. Performing nodal slides that shorten all the cuts to
a limit point, pushing the nodes all the way to the boundary, corresponds to a
degeneration from $\CP^2$ to the weighted projective space $\CP(1,4,25)$.
Following the degeneration, $T(1,4,25)$ goes to the `central' fiber of the
standard base diagram of $\CP(1,4,25)$.

The projective plane admits degenerations to weighted projective spaces
$\CP(a^2,b^2,c^2)$, where $(a,b,c)$ is a Markov triple, i. e., satisfies the
Markov equation: 

\begin{equation} \label{Markov}
a^2 + b^2 + c^2 =3abc.
\end{equation}

For each $\CP(a^2,b^2,c^2)$, one can associate a monotone Lagrangian torus, $T\abc$, in
$\CP^2$ in either of the following ways: 

\begin{itemize}
 
 \item[-] by following the necessary nodal trade, nodal slide and
          \emph{transferring the cut} - see Definition \ref{def_transcut}
          - operations until we get to a base diagram that is about to
          degenerate to the base of the moment map for the standard torus action
          on $\CP(a^2,b^2,c^2)$ and considering the monotone fiber - see section
          \ref{abc}, Proposition \ref{transcut};
            
\item[-] by performing three \emph{rational blowdown} surgeries on $\CP\abc$ -
         see section 10 of \cite{MS} - that replace a small neighbourhood of
         each point mapping to the vertex of the moment polytope of $\CP\abc$,
         having a lens space of the form $L\kl$ as its boundary by a rational
         ball having the same boundary - see Figure \ref{figRatBD} - and
         considering the monotone fiber. 
         
\end{itemize}

\begin{figure}[h]
  \begin{center}
\centerline{\includegraphics[scale=0.7]{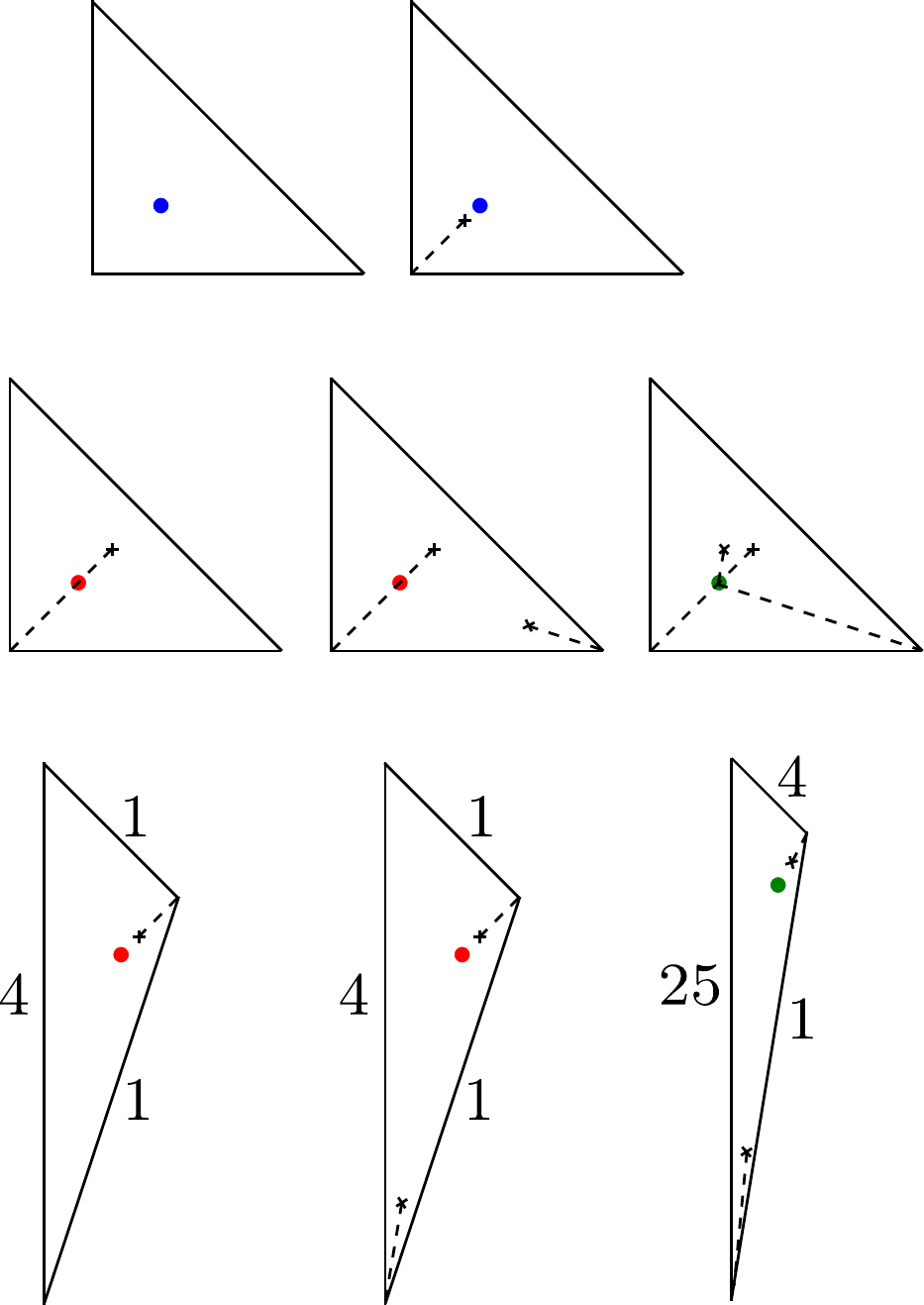}} 
 \caption{The procedure for going from the Clifford
 torus on the top left base diagram, to the Chekanov torus (third base 
 diagram) and to the $T(1,4,25)$ torus (fifth base diagram) by applying nodal
trades and nodal slides - see \cite{RV,MS,NLMS} for definitions. The dots represent the 
image of the monotone tori in the base diagrams. 
Each of the bottom diagrams is equivalent to the one right
 above it since they are related by transferring the cut operations
 - see Definition \ref{def_transcut}.}
\label{figIntro}
\end{center}
\end{figure}

 We will prove:

\begin{theorem} \label{mainthm}
  
If $(a,b,c)$ and $(d,e,f)$ are two distinct Markov triples then the monotone
Lagrangian tori $T(a^2, b^2, c^2)$ and $T(d^2, e^2, f^2)$ are not Hamiltonian
isotopic.
  
\end{theorem}

In \cite{RV}, we gave an explicit description of $T(1,4,25)$. We first predicted
the number of Maslov index 2 discs each $T(1,4,25)$ bounds, by applying
wall-crossing mutations to the superpotential, as described by Galkin and Usnich
in \cite{GaUs} - see also sections 2.4 and 3 of \cite{RV}. But unfortunately
wallcrossing formulas are not proved to hold yet. That forced us to directly
compute all the Maslov index 2 holomorphic discs $T(1,4,25)$ bounds.

\begin{figure}[h!]
  \begin{center}
\centerline{\includegraphics[scale=1]{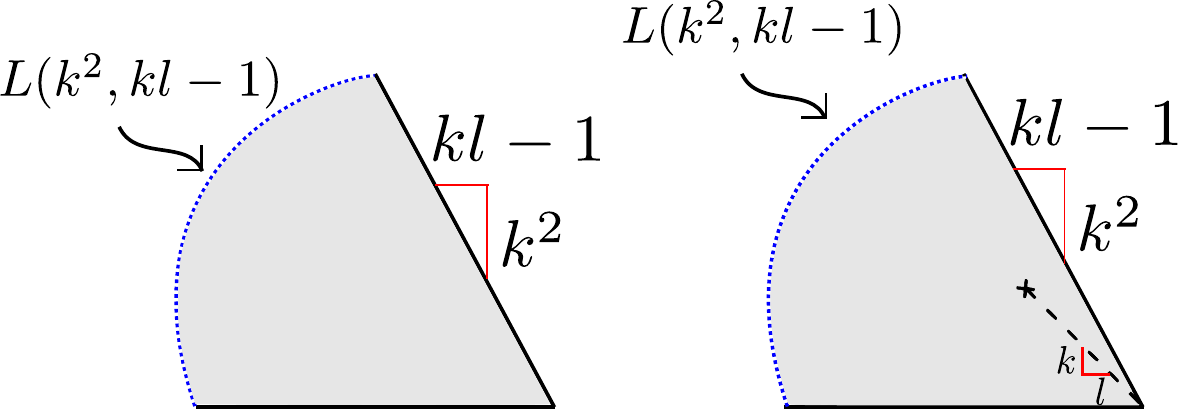}} 
 \caption{The picture on the left represents the base of a toric neighbourhood of
an orbifold point. The picture on the right is the base of an almost toric
fibration on a rational ball having the boundary the lens space $L\kl$. The lens
space $L\kl$ is the union of the fibers over the dotted component of the
boundary of the base diagrams.}
\label{figRatBD}
\end{center}
\end{figure}

In this paper, we employ the technique of neck-stretching from symplectic field
theory. We use it to find restrictions on the relative homotopy classes in
$\pi_2(\CP^2, T\abc)$ that can be represented by a holomorphic disc with Maslov
index 2. More precisely, we describe the convex hull of all classes in
$\pi_2(\CP^2,$ $T\abc)$, represented by Maslov index 2 holomorphic discs. It
follows directly from the work of Gromov \cite{Grom} that one can construct
Hamiltonian isotopy invariants for monotone Lagrangian submanifolds from
algebraic counts of holomorphic discs with Maslov index 2 (Theorem 6.4 of
\cite{RV}). This is a well known fact in the symplectic geometry community and
was inferred in Proposition 4.1.A of \cite{EP}. Using this invariants we are able to
distinguish the tori $T\abc$ for different Markov triples.

This is the first example of infinitely many Lagrangian isotopic but not
Hamiltonian isotopic monotone Lagrangian tori living in a compact symplectic
manifold. A similar result in $\R^6$ was given by Auroux in \cite{DA-R6}.

\begin{remark} 
  
In \cite{WW}, Wu used neck-stretching technique to compute
holomorphic discs bounded by his torus arising as a `central' fiber from a
semi-toric system on $\CP^2$. By the description of the Chekanov torus as
$T(1,1,4)$ given in \cite{RV}, it was suggestive to us that this torus is a different 
presentation of the Chekanov torus. A proof that Wu's torus, among others,
is a presentation of the Chekanov torus is given by Oakley and Usher in \cite{UO}. 
 
 \end{remark}

\begin{remark}
While writing this paper the author learned that Galkin and Mikhalkin have 
independently obtained the same result - \cite{GaMi}.
\end{remark}

The paper is organized as follows. 

In section \ref{abc} we show that the boundary of a neighbourhood of an orbifold
point in $\CP\abc$ is contactomorphic to a lens spaces of the form $L(k^2,kl -
1)$ as their boundaries. Hence we can apply rational blowdown on these
neighbourhoods, as in section 10 of \cite{MS}. We show that, after applying the
rational blowdowns, we obtain an almost toric fibration of $\CP^2$. This is done
by showing that we can get to the same almost toric fibration by performing a
series of nodal trade, nodal slide and transferring the cut operations to the
standard moment polytope of $\CP^2$.

In section \ref{SFT} we describe a technique originating in symplectic field
theory, often called neck-stretching. In the subsection \ref{split}, we give a
quick review of neck-stretching, also known as splitting of a symplectic
manifold along a contact hypersurface - see \cite{EGHsft}, \cite{BEHWZsft},
\cite{WW}. In subsection \ref{ACS}, we define what kind of almost complex
structures are adjusted for the neck-stretching we perform. In section
\ref{Examp}, we work out an example of neck-stretching that is important for the
proof of Theorem \ref{mainthm}. In subsection \ref{CompThm}, we state, from
\cite{BEHWZsft} and \cite{EGHsft}, the main compactness theorem of
pseudo-holomorphic curves for neck-stretching.

Section \ref{proof} is devoted to the proof of Theorem \ref{mainthm}. We use the
technique of neck- stretching to describe the convex hull of all classes in
$\pi_2(\CP^2,$ $T\abc)$, represented by Maslov index 2 holomorphic discs. The
proof then follows from Theorem 6.4 of \cite{RV} (Lemma \ref{lem_Thm6.4ofRV}),
which is an immediate consequence of the work of Gromov \cite{Grom} - see also
proposition 4.1 A of \cite{EP}.
 
\textbf{Acknowledgments.} I am extremely grateful to Denis Auroux for huge
support and invaluable discussions. Also, I want to thank Weiwei Wu for useful
discussions during my visit to Michigan State University. This work was
supported by the CNPq - Conselho Nacional de Desenvolvimento Cient\'ifico e
Tecnol\'ogico, Ministry of Science, Technology and Innovation, Brazil; the
Department of Mathematics of University of California at Berkeley; the National
Science Foundation grant number DMS-1264662; and (during the revision of the
paper) the Herchel Smith Postdoctoral Fellowship - Dep. of Pure Mathematics and
Mathematical Statistics - University of Cambridge.

\section{Degenerations to $\CP\abc$ and almost toric fibrations} \label{abc}

In this section, we show that, for each Markov triple $(a,b,c)$, there is 
a monotone Lagrangian torus $T\abc$, which is the `barycentric fiber' depicted
in a base diagram of an almost toric fibration. For a detailed account on
almost toric fibrations we refer the reader to the work of Symington \cite{MS}
and Leung-Symington \cite{NLMS}.

First, we will define an operation we call \emph{transferring the cut}
(Definition \ref{def_transcut}). A vector $v$ in a lattice
$\Lambda$ is called a \emph{primitive vector}, if it is not a positive multiple
of another vector in the lattice. If $w = \lambda v \in \Lambda \otimes \R$, with
$v \in \Lambda$ primitive and $\lambda \in \R_{\ge 0}$, we say that $w$ has
\emph{affine length} $\lambda$. Transferring the cut operation changes a base
diagram in $\R^2$, whose edges have affine lengths $a^2, b^2, c^2$, into another
base diagram whose edges have affine lengths $(3bc - a)^2, b^2, c^2$
(Proposition \ref{transcut}). These will represent the same almost toric
fibration of $\CP^2$ (see Figure \ref{figABCtoA'}).
 
We recall that Markov triples are obtained from
(1,1,1) by a sequence of `mutations' of the form 
\begin{equation} \label{MarkMut}
  (a , b , c) \rightarrow (a' = 3bc - a , b , c). 
  \end{equation}

Hence, we show the claim of \cite{RV} that an almost toric fibration having 
$T\abc$ as its central fiber can be obtained from the moment polytope of the 
standard torus action on $\CP^2$ by a series of nodal trade, nodal slide and
transferring the cut operations.

Along the way, we show that the boundary of a neighbourhood of an orbifold point
in $\CP\abc$ is a lens space of the form $L(k^2,kl - 1)$. It follows then that
we obtain an almost toric fibration of $\CP^2$ with $T\abc$ as the central fiber
by performing three rational blowdown operations on small neighbourhoods of the
corners of the standard base diagram of $\CP\abc$ - see section 10 of \cite{MS}
and Figure \ref{figRatBD}. 
  
Following the notation of section 5 of \cite{MS}, let $(B, \A, S)$ be a almost
toric base of some almost toric fibration, where $B$ is the base of the singular
Lagrangian fibration, $\A$ is the induced affine structure, $S =
\bigcup_{i=1}^{N} \{s_i\}$ and $s_i \in B$ are the nodes. Denote by $\A_0$ the
standard affine structure in $\R^2$. Recall that, for $b$ close to a node $s_i$,
there is an \emph{eigendirection} in $T_b B$ invariant under a monodromy around
each $s_i$. An \emph{eigenline} through $s_i$, is the maximal affine linear
immersed one manifold tangent to eigendirection of $s_i$ at each $T_b B$. An
\emph{eigenray} is one of the two components of an \emph{eigenline} minus the
respective node. See definition 4.11 of \cite{MS}.

Consider a base diagram $\Delta \subset \R^2$ (assume it is connected), which is
the image of an affine embedding $\Phi$ of $(B \setminus \bigcup_{i=1}^{N} R_i,
\A)$ into $(\R^2, \A_0)$, where $R_i$ is an (oriented) eigenray leaving the node
$s_i$. Denote $R_1$ by $R^+$, the eigenline containing $R^{+}$ by $L$, and by
$R^{-} \subset L$ the eigenray opposite to $R^+$. Let $P^{\pm} = \{ \lim_{x \to
y} \Phi(x) | x \in B ; y \in R^{\pm}\}$, be the \emph{branch locus} of
$R^{\pm}$. We have that the branch locus of $L$, $P^+ \cup \{s_1\} \cup P^-$,
divides the base diagram $\Delta$ into two components, $\Delta_l$ (left) and
$\Delta_r$ (right).

\begin{figure}[h!]   
  
\begin{center} 
\centerline{\includegraphics[scale=0.7]{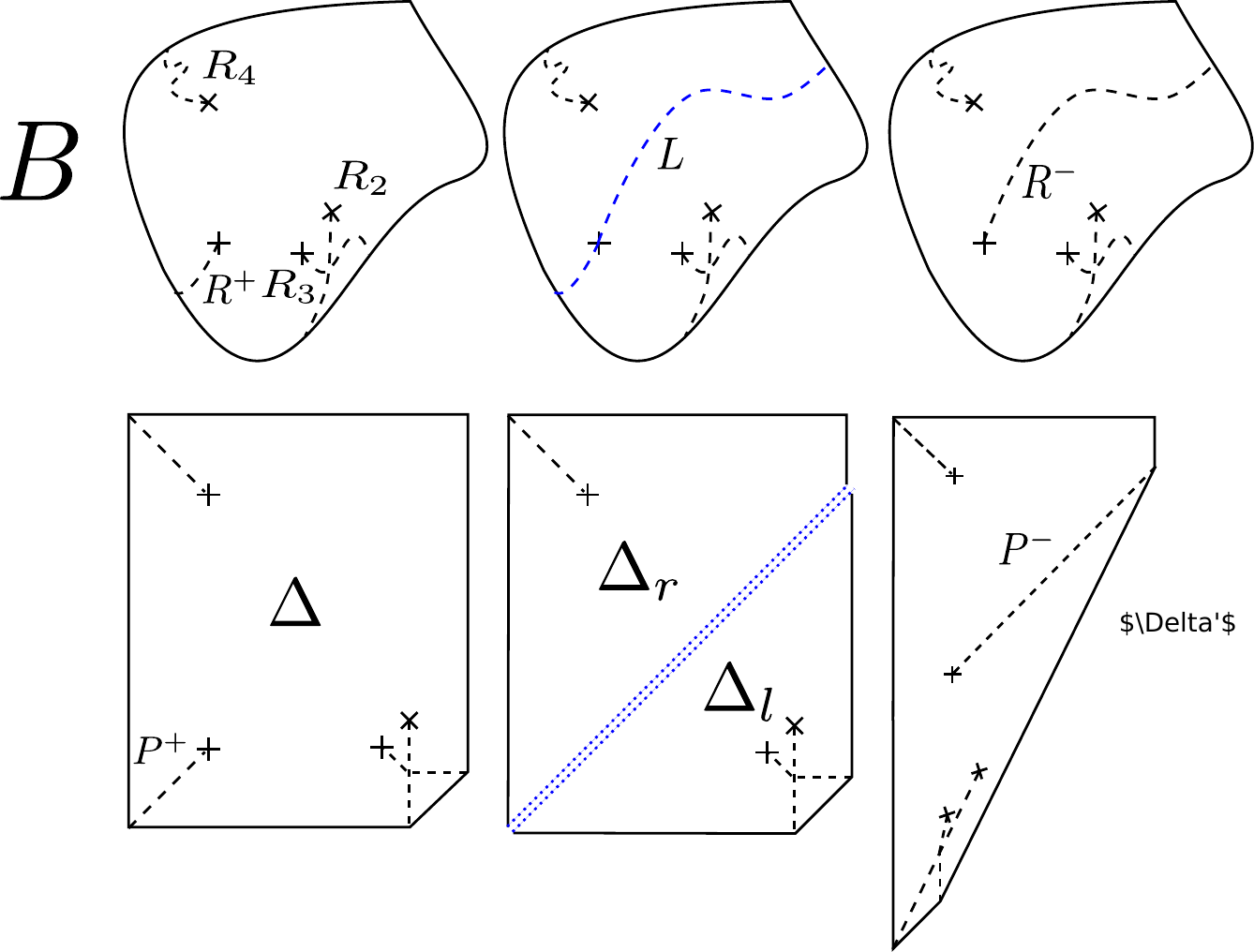}} 
\caption{The top three pictures represents a base $B$ of an almost toric
fibration with different set of eigenrays, represented by the dashed lines. Each
one of the bottom three pictures is the image of an affine embedding of $B$
minus some eigenrays (the ones represented on the picture right above) into $\R^2$.
By applying \emph{transferring the cut} operation on the left of $R^+$ to the
base diagram $\Delta$ (bottom left picture), we obtain the base diagram
$\Delta'$ (bottom right picture, after rescaling for visual purpose).}
\label{figtranscut} 
\end{center} 
\end{figure}  

We will construct a new base diagram $\Delta'$, corresponding to an affine
embedding $\Phi'$ of $(B \setminus R^{-}\cup \bigcup_{i=2}^{N} R_i, \A)$ into
$(\R^2, \A_0)$, representing the same almost toric fibration as $\Delta$. Let
$M^{lr}$ be the monodromy used to go from $\Delta_l$ to $\Delta_r$, through
$P^+$. Essentially, $\Delta'$ is obtained by gluing at $P^+$, $M^{lr}(\Delta_l)$
to $\Delta_r$ , where the monodromy $M^{lr}$ is applied centred at $s_1$. In
other words, $\Phi'$ is equal to $\Phi$ on $\Phi^{-1}(\Delta_r)$, to
$M^{lr}\circ \Phi$ on $\Phi^{-1}(\Delta_l)$ and extends continuously to $R^+$,
so that $\Phi'(R^+) = P^+$.

 \begin{definition}\label{def_transcut} The base diagram $\Delta'$ constructed
 above is said to be obtained from $\Delta$ by a \emph{transferring the cut}
 operation on the left of $R^+$. 
 
 \end{definition}

 The definition of \emph{transferring the cut} operation on the \emph{right} of
 $R_1$ is obtained in a totally analogous way. From now on, we will abuse
 notation, as we also denote by $R^+$ its branched cover $P^+$.

 Consider now the standard moment polytope of $\CP\abc$, for $(a,b,c)$ Markov triple,
 with oriented edges $a^2\bu_1$, $b^2\bu_2$, $c^2\bu_3$, as in the left picture
 of Figure \ref{figABC}. We can arrange $\bu_1 = (b^2 ,- m_1)$, $\bu_2 = - (a^2, m_2)$,
 $\bu_3 = (0,1)$.

\begin{figure}[h!]
  \begin{center}
\centerline{\includegraphics[scale=0.7]{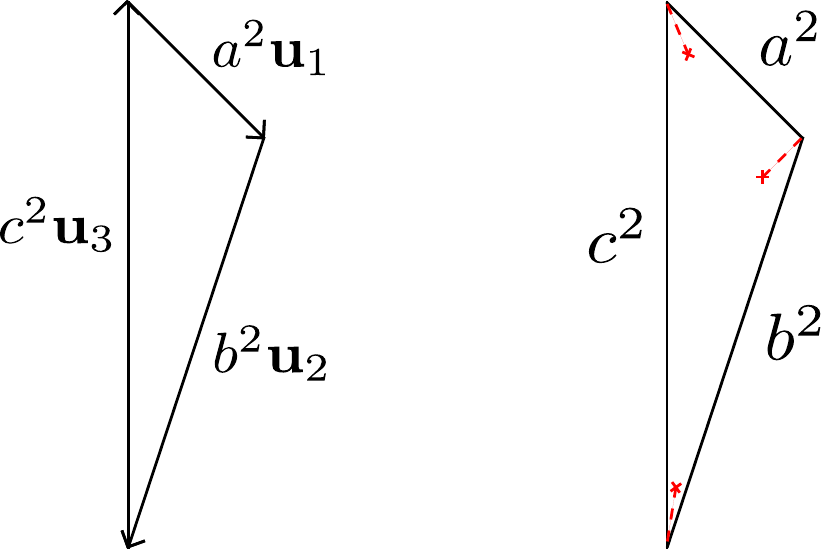}} 
 \caption{The left picture is the moment polytope for the standard torus action on
 $\CP\abc$. The right picture is obtained form the left one by three rational
 blowdown operations on the neighbourhood of each vertex.}
\label{figABC}
\end{center}
\end{figure}    

 \begin{proposition} \label{mili}
   The positive integers $m_1$, $m_2$ are of the form $bl_1 - 1$, $al_2 - 1$, 
   respectively, $l_1, l_2$ in $\Z_{>0}$. Hence, the boundary of a neighbourhood of the vertex opposite to
   $a^2\bu_1$, respectively $b^2\bu_2$, is a lens space of the form $L(a^2, al_2 - 1)$,
   respectively $L(b^2, bl_1 - 1)$.  
 \end{proposition}
 
 \begin{proof}
  First we note that $(a,b,c)$ are mutually co-prime. In fact, if $p$ divides
  two of them, by the Markov equation \eqref{Markov}, it must divide the third one. The 
  numbers $a' = 3bc -a$, $b' = 3ac - b$ and $c' = 3ab - c$ are also divisible by $p$.
  Since we can reduce any Markov triple to $(1,1,1)$ by applying mutations of the form
  \eqref{MarkMut}, we must have $p = 1$.
  
  By equating the last coordinate of $a^2\bu_1 + b^2\bu_2 + c^2\bu_3 = 0$ and using
  the Markov relation \eqref{Markov} we get
  
  \begin{eqnarray}
    a^2m_1 + b^2m_2 & = & c^2, \\ \label{eqmi1}
    a^2(m_1 + 1) + b^2(m_2 + 1) & = & 3abc . \label{eqmi2}
  \end{eqnarray}
    
 Working modulo $a$ and modulo $b$, we must have $m_1 = bl_1 - 1$, $m_2 = al_2 - 1$. 
 Positivity of $l_1 , l_2$ follows from positivity of $m_1, m_2$.
 
 The second statement of the Proposition follows immediately from section 9.3 of 
 \cite{MS}.

 \end{proof}
 
By applying an appropriate $SL(2, \Z)$ transformation to the base diagram, 
sending $\bu_2$ to $(0,1)$, allow us to conclude, using the above Proposition,
that the remaining vertex has a neighbourhood with boundary a lens space of
the form $L(c^2, cl_3 - 1)$. Hence, we can apply rational blowdown operations in
a neighbourhood of each vertex. We get from $\CP\abc$, represented by its
standard moment polytope, to the almost toric fibration represented by the right
picture of Figure \ref{figABC}.
 
 \begin{remark}
   Consider the primitive vectors $\bw_1 = - (a, l_2)$, $\bw_2 = (-b, l_1)$ and
   $\bw_3$ representing the cuts respectively opposite to the edges
   $a^2\bu_1$, $b^2\bu_2$,
   $c^2\bu_3$. The reader can verify that 
   
   \begin{equation*}
     \frac{ac\bw_2 - bc\bw_1}{3} = c^2\bu_3; \ \ \frac{bc\bw_1 - ab\bw_3}{3} = b^2\bu_2; \ 
   \ \frac{ab\bw_3 - ac\bw_2}{3} = a^2\bu_1. 
   \end{equation*}
 
 This shows that the lines leaving the vertices in the direction of the
 respective cuts intersect in a common point (where the monotone fiber lies).
 This point is the weighted barycenter of the triangle, i. e., the center of
 mass of a system with weights $a^2$, $b^2$, $c^2$ on the vertices respectively
 opposite to the edges $a^2\bu_1$, $b^2\bu_2$, $c^2\bu_3$. To see this, the reader 
 only needs to check that
 
 \begin{equation*}
   a^2\frac{bc}{3}\bw_1 + b^2\frac{ac}{3}\bw_2 + c^2\frac{ab}{3}\bw_3 =
   \frac{abc}{3}( a\bw_1 + b\bw_2 + c\bw_3) = 0.
 \end{equation*}
 
 \end{remark}
 
 The remaining part of this section is devoted to prove Proposition
 \ref{transcut}, from which we deduce that the space obtaining after
 performing rational blowdowns in a neighbourhood of each (point mapped to each)
 vertex of the standard moment polytope of $\CP\abc$ is $\CP^2$.
  
 \begin{proposition} \label{transcut}
   Consider the diagram on the right of
 Figure \ref{figABC}, with edges $a^2\bu_1$, $b^2\bu_2$, $c^2\bu_3$, and the 
 cut $R^+$ opposite to $a^2\bu_1$. By applying transferring the cut operation
 to the left of $R^+$, we obtain a diagram so that the affine lengths of the 
 edges is a constant multiple of $c^2$, $b^2$, $a'^2$, where $a' = 3bc - a$.  
 \end{proposition}
  
\begin{figure}[h!]
\begin{center}
 \begin{minipage}[b]{0.5\linewidth}
\centerline{\includegraphics[scale=0.85]{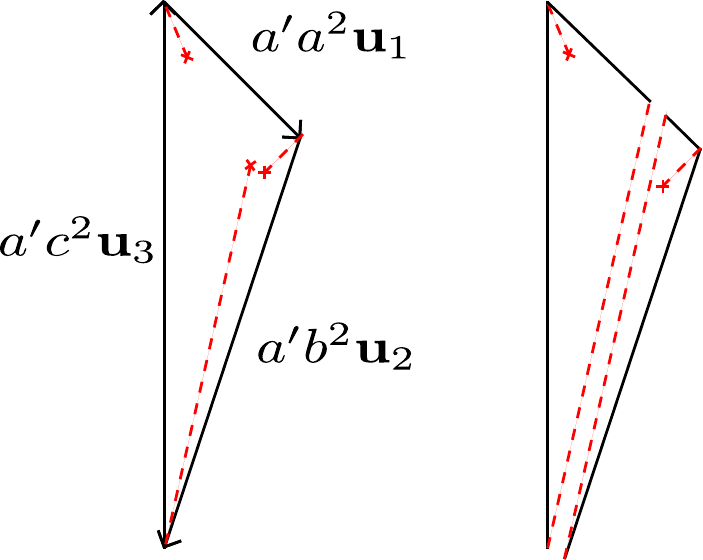}} 
\end{minipage}
 \begin{minipage}[b]{0.35\linewidth}
\centerline{\includegraphics[scale=0.85]{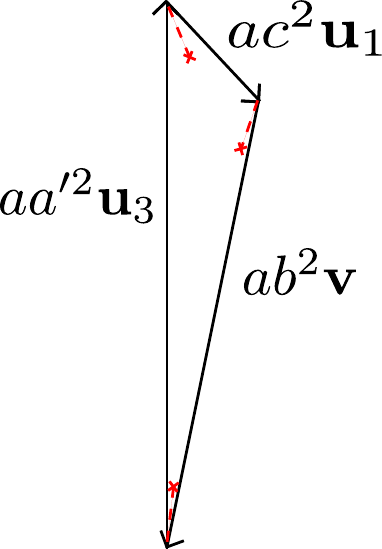}} 
 \end{minipage}
\end{center}
\caption{Transferring the cut operation to the left of the cut $R^+$. We
multiplied by a factor of $a'$ to simplify computations and lengthened the cut
so that the resulting diagram contains $T(c^2,b^2,a'^2)$ as its barycentric
fiber.}
\label{figABCtoA'}

\end{figure} 

\begin{proof}
  We first multiply all the edges by a factor of $a'$ in order to make the 
  computations simpler. We cut the edge parallel to $\bu_1$ at a length $\alpha$
  so that, for $\bw_1 = -(a,l_2)$ the vector representing the eigenray $R^+$, we have
  
  \begin{equation} \label{eqalpha1}
    a'c^2\bu_3 = - \beta \bw_1 - \alpha \bu_1
  \end{equation} 
   
   Using that $\bu_1 = (b^2 ,-(bl_1 - 1))$ - see Proposition \ref{mili} - we have
   \begin{eqnarray}
     0 & = & a\beta - b^2 \alpha;  \label{eqalpbeta1} \\
     a'c^2 & = & \beta l_2 + (bl_1 - 1)\alpha. \label{eqalpbeta2} 
   \end{eqnarray}

 From equation \eqref{eqmi2} and $m_1 + 1 = bl_1$, $m_2 + 1 = al_2$, we get
 
 \begin{equation} \label{eql1l2}
   3c = bl_2 + al_1.
 \end{equation}
   
 Using \eqref{eql1l2} and \eqref{eqalpbeta1} in \eqref{eqalpbeta2}, and recalling
 that $a' = 3bc -a$, we get
 
 \begin{equation}
   a'c^2 = (\frac{b^2}{a}l_2 + bl_1 - 1)\alpha = \frac{1}{a}( b( b l_2 + a l_1) - a)\alpha
  = \frac{a'}{a} \alpha. \label{eqalpha2}
 \end{equation}

 Hence we cut at $ac^2 \bu_1$. Now we apply the monodromy through $R^+$ from
 left to right (recall that $R^+$ is oriented pointing away from the node),
 which sends $\bu_2$ to $\bu_3$ and fixes $\bw_1$. After regluing, the vertical
 edge has length $a'(c^2 + b^2) = a(a')^2$, since $a \cdot a' = b^2 + c^2$ is
 another way to express the Markov relation \eqref{Markov}. We only need to show
 that the remaining edge represented by the vector $-a(a')^2\bu_3 - ac^2\bu_1$
 has affine length $ab^2$.
 
We have that
\begin{equation} \label{lambv}
  -a(a')^2\bu_3 - ac^2\bu_1 = -a \cdot (c^2b^2 , (a')^2 + c^2m_1).
\end{equation}

Since $c$ and $a'$ are co-primes, we only need to show that $b^2$ divides $(a')^2 + 
c^2m_1$. From equation \eqref{eqmi1} and $a \cdot a' = b^2 + c^2$, we get that

\begin{equation}
  a^2m_1 + c^2 \equiv 0 \mod b^2, \ \ \text{and} \ \  a \cdot a' \equiv c^2 \mod 
  b^2. 
\end{equation}

Hence,

\begin{equation}
  (a')^2 + c^2m_1 \equiv c^4a^{-2} + c^2(-c^2a^{-2}) \equiv 0 \mod b^2.
  \end{equation}

 \end{proof}  
  
  It is clear that considering another cut or transferring the cut operation to
  the left gives an analogous result. Recall that any given Markov triple
  $(a,b,c)$ can be obtained from $(1,1,1)$ by a sequence of mutation operations
  \eqref{MarkMut}. Therefore, one can apply a series of nodal trades, nodal
  slides and transferring the cut operations to the standard moment polytope of
  $\CP^2$, scaled by a factor of $abc$, to get to the almost toric fibration,
  represented by the diagram on the right of Figure \ref{figABC}, containing
  $T\abc$ as the monotone fiber.
  
  \begin{corollary} \label{CorTransCut}
  
   Perform three rational blowdowns on small neighbourhoods of (the pre-image of)
   each vertex of the standard moment polytope of $\CP\abc$, bounded by lens
   spaces of the form $L(a^2,al_1 -1), L(b^2, bl_2 -1), L(c^2,cl_3 - 1)$. We
   then obtain an almost toric fibration of $\CP^2$ as depicted in the right
   picture of Figure \ref{figABC}.
  
  \end{corollary}
  
  \begin{remark}
    The symplectic form of the almost toric fibration of $\CP^2$ represented by 
    the right base diagram of Figure \ref{figABC} equivalent to the symplectic 
    form of the standard moment polytope of $\CP^2$ scaled by a factor of $abc$. 
  \end{remark}
  
\section{Neck Stretching - SFT} \label{SFT}

In this section we discuss a technique coming from symplectic field theory,
often called neck-stretching. It is a way of splitting a symplectic manifold
along a contact hypersurface in which we stretch a neighbourhood of the contact
hypersurface until it reaches a limit where it splits apart. Compactness results tell us
what happens to the limit of pseudo-holomorphic curves after we split the
symplectic manifold. We refer the interested reader to \cite{EGHsft}, 
\cite{BEHWZsft}. In \cite{WW}, Wu also gives a quick review on neck-stretching.

Our idea is to apply these techniques to the lens spaces described on the
previous section. More precisely, to the boundaries of rational balls which are
neighbourhoods of the singular fibers on an almost toric fibration containing
$T\abc$ as the central fiber, depicted in the right diagram of Figure
\ref{figABC} - see also Figure \ref{figRatBD}. See section 9 of \cite{MS}, for
understanding how to see the respective lens spaces as contact manifolds.

\subsection{Splitting} \label{split}
Let $V$ be a hypersurface of contact type in a symplectic manifold $(M, 
\omega)$. This means that, in a neighbourhood of $V$, one can define a Liouville 
vector field $X$ (so the Lie derivative $\mathcal{L}_X \omega = d\iota_X \omega = 
\omega$), transversal to $V$, for which $\alpha = \iota_X \omega$ restricted to $V$ is
a contact form.

Following the notations of \cite{EGHsft}, let us assume that $V$ divides $M$ in
two components $M_+$ and $M_-$, as is the case for each lens space described in
section \ref{abc}. We choose $M_+$ and $M_-$ so that $X$ points inwards along
$M_+$, and outwards along $M_-$. Hence we can complete $M_+$ and $M_-$ by gluing
along $V$ different halves of its symplectization $(V \times \R, d(e^t\alpha))$,
matching $X$ with $\frac{\del}{\del t}$, obtaining 

\begin{equation}\label{m-inf}
    (M_{-}^{\infty}, \omega_{-}^{\infty}) = (M_-, \omega)\cup (V\times [0, \infty), 
  d(e^t\alpha))
\end{equation}

and

\begin{equation}\label{m+inf}
    (M_+^\infty, \omega_+^\infty) = (M_+, \omega)\cup (V\times (-\infty, 0], 
  d(e^t\alpha)). 
\end{equation}
 
We also consider partial completions 
 
\begin{equation}\label{m-tau}
    (M_{-}^{\tau}, \omega_{-}^{\tau}) = (M_-, \omega)\cup (V\times [0, \tau], 
  d(e^t\alpha))
\end{equation}

and

\begin{equation}\label{m+tau}
    (M_+^\tau, \omega_+^\tau) = (M_+, \omega)\cup (V\times [-\tau, 0], 
  d(e^t\alpha)). 
\end{equation}

Now we note that $(V\times [-\tau, 0], e^\tau d(e^t\alpha)) = (V\times [0, \tau],
d(e^t\alpha))$ and \\$(V\times [0, \tau], e^{-\tau} d(e^t\alpha)) = (V\times [-
\tau, 0], d(e^t\alpha))$. So, we see that $(M_{-}^{\tau}, e^{-\tau}\omega_{-}^{\tau})$,
$(V\times [- \tau, \tau], d(e^t\alpha))$ and $(M_+^\tau, e^\tau \omega_+^\tau)$ fit 
together to give a symplectic manifold $(M^\tau, \omega^\tau)$. We say that we 
inserted a neck $V\times [- \tau, \tau]$ of length $2\tau$ in between $M_+$ and $M_-$.

We see that in the limit we have $M^\infty = M_-^\infty \cup M_+^\infty$ - see 
section 1.3 of \cite{EGHsft}, especially Figure 1. We 
note that $\omega^\tau$ goes to zero in one end, while it blows up in the other. 
For our purpose, we will be more focused in what happens on $M_+$, so we 
consider a stretching $(M^\tau, e^{-\tau}\omega^\tau)$, so that the symplectic 
form converges to $0$ in $M_-^\infty$ and to $\omega_+^\infty$ in $M_+^\infty$.

\subsection{Almost complex structures - compatible and adjusted} \label{ACS}

For a symplectic manifold with cylindrical ends, we require some other
properties for an almost complex structure $J$ to be said compatible. Besides
the usual compatibility conditions with the symplectic form, we say that $J$
is \emph{compatible} if at any
cylindrical end of the form $(V \times [0,\infty), d(e^t\alpha))$ or
$(V\times(-\infty,0], d(e^t\alpha))$, positive or negative, we have that

\begin{itemize}
  \item[-] $J$ is invariant with respect to translations $t \mapsto t \pm a$, 
  $a>0$;
  \item[-] the contact structure $\xi = \{\alpha = 0\}$ is invariant under $J$;
  \item[-] $J\frac{\del}{\del t} = R_\alpha$, where $R_\alpha$ is the Reeb vector
  field associated with $\alpha$ - see section 1.2 of \cite{EGHsft} for definiton
  of Reeb vector field.   
\end{itemize}

%So, regarding to the splitting, to ensure that we end up with compatible almost
%complex structures for $(M_+^\infty,\omega_+^\infty)$ and $(M_-^\infty, 
%\omega_-^\infty)$ - \eqref{m+inf}, \eqref{m-inf} - 

We say that an almost complex structure $J$ on $M$ is \emph{adjusted} for the
splitting situation if 

\begin{itemize}
 
  \item[-] on $V$, the contact structure $\xi = \{\alpha = \iota_X \omega = 0\}$
   is invariant under $J$, and
  
  \item[-] $JX = R_\alpha$, where $X$ is (a multiple of) the Liouville vector field
           defining $\alpha$ and $R_\alpha$ is the Reeb vector field associated
           with $\alpha$.  
   
\end{itemize}
 
 Given an adjusted $J$, we can define $J^\tau$ on $M^\tau$ by setting it equal
 to $J$ on $M_+$ and $M_-$ and requiring it to be invariant under translation on
 $V\times[-\tau,\tau]$. So, when $\tau \to \infty$, we end up with compatible
 almost complex structures $J_+^\infty$ on $(M_+^\infty,\omega_+^\infty)$ and
 $J_-^\infty$ on $(M_-^\infty, \omega_-^\infty)$. 

\subsection{Example}\label{Examp}

We consider the following example because it is going to be important in the proof of 
Theorem \ref{mainthm}.

Consider $\CX^2$ with the Fubini-Study symplectic form $\omega = i/2(dz_1\wedge
d\bar{z}_1 + dz_2\wedge d\bar{z}_2)$ and $S^3(2) = \{ |z_1|^2 + |z_2|^2 = 4 \}$.
We have that the radial vector field $X = \frac{1}{2}(z_1, z_2)$ is a Liouville
vector field. In fact, $\alpha := \iota_{X} \omega = i/4(z_1d\bar{z}_1 -
\bar{z}_1 dz_1 + z_2d\bar{z}_2 - \bar{z}_2dz_2)$ and $\Lag_{X}\omega =
d\iota_{X} \omega = \omega$.

Now we see that the standard complex structure is adjusted. First one can check
that $\xi = \{\alpha = 0\} = TS^3(2) \cap i \cdot TS^3(2)$, which is formed by the
vectors that are orthogonal to both $X$ and $iX$, with respect to the Euclidean
metric $<\cdot,\cdot> = \omega(\cdot, i\cdot)$. Hence the first condition is
satisfied and $d\alpha(iX, \cdot) = \omega_{|TS^3(2)}(iX, \cdot) = <X,
\cdot>_{|TS^3(2)} = 0$. Also, $\alpha(iX) = \frac{|z_1|^2 + |z_2|^2}{4} = 1$.
Therefore, $iX = R_\alpha$, the Reeb vector field associated with $\alpha$,
showing that the complex structure given by multiplication by $i$ is adjusted for
the splitting with regards to $V = S^3(2)$.

Let's now look at $(M_+^\infty, \omega_+^\infty, J_+^\infty)$ - see \ref{m+inf}
- for $M = \CX^2$, $V = S^3(2)$, $M_+ = \CX^2 \setminus B(2)$ and considering
the standard Fubini-Study form and complex structure. 

\begin{claim} \label{claimExamp}
  After splitting, $(M_+^\infty, \omega_+^\infty, J_+^\infty)$ is a K\"ahler manifold
 isomorphic to $(\CX^2 \setminus \{ 0 \}, \omega_{FS}, i)$.
 
 \end{claim}

\begin{proof}
  
  We only need to show that the following embedding gives an biholomorphic 
  symplectomorphism between $(S^3(2) \times (-\infty, 0], d(e^t\alpha), 
  J_+^\infty)$ and the punctured ball $(B(2) \setminus \{ 0 \}, \omega_{FS}, i)$. 
 
 \begin{eqnarray} \label{Exampemb}
 \phi : S^3(2) \times (-\infty, 0] &\longrightarrow& \CX^2 \setminus \{ 0 \}   \nonumber \\
((z_1, z_2), t) & \mapsto & (e^{\frac{t}{2}}z_1, e^{\frac{t}{2}}z_2)
\end{eqnarray}
 
 We see that, at $((z_1, z_2), t)$, $d\phi (\frac{\del}{\del t}) =
 \frac{1}{2}(e^{\frac{t}{2}}z_1, e^{\frac{t}{2}}z_2) = X \circ \phi$.
 
 Take a vector $v \in T\CX^2$ at a point $(e^{\frac{t}{2}}z_1,
 e^{\frac{t}{2}}z_2)$. We can write $$v = e^{t/2}(u + aX + bR),$$ where $R =
 iX$ and $u \in T_{(e^{t/2}z_1, e^{t/2}z_2)} S^3(2e^{t/2}) \cap i \cdot T_{(e^{t/2}z_1,
 e^{t/2}z_2)} S^3(2e^{t/2}) \cong T_{(z_1, z_2)} S^3(2) \cap i \cdot T_{(z_1, z_2)}
 S^3(2)$. So, $v = d\phi (u + bR, a\frac{\del}{\del t})$.
 
 Then, recalling that $J_+^\infty \frac{\del}{\del t} = R$, we have that
 
 $$ (\phi_* J_+^\infty)\cdot v = d\phi (J_+^\infty(u + bR, 
 a\frac{\del}{\del t})) = d\phi(iu + aR, -b \frac{\del}{\del t}) =
  e^{t/2}(iu - bX + aR) = iv .$$
  
  Therefore, $\phi_* J_+^\infty = i$. We leave to the reader to check that 
  $\phi^* \omega_{FS} = d(e^t\alpha)$.

\end{proof}

The same result holds with a sphere of different radius. We only have to glue the
infinite neck using a multiple of the Liouville vector field $X$ to obtain,
after multiplication by $i$, the Reeb vector field.

Consider now the lens space $L(n,m)$ inside $\CX(n,m) := \CX^2 / (z_1,z_2) \sim
(e^{2\pi i / n}z_1, e^{2\pi i m /n}z_2)$ as the quotient of $S^3 = \del B(r)$,
for some fixed $r > 0$. We consider in $L(n,m)$ the contact structure induced
from the one in $S^3$, which is invariant under the $\Z/n\Z$ action used for the
quotient. Taking the standard complex and symplectic structures in $\CX(n,m)$
coming from $\CX^2$, and contact structure on $V = L(n,m)$ we obtain an
analogous result:

\begin{corollary}\label{corExamp}
  
  Using the same notation for the above setting, we have that, after splitting
  along $V = L(n,m)$, $(M_+^\infty, \omega_+^\infty, J_+^\infty)$ is a K\"ahler
  manifold isomorphic to $(\CX(n,m) \setminus \{ 0 \}, \omega_{Std}, i)$.
  
  \end{corollary} 

\subsection{Compactness Theorem} \label{CompThm}

Here we state a version of Theorem 1.6.3 of \cite{EGHsft} adapted to our
situation - see also Theorem 10.6 of \cite{BEHWZsft}. Consider a symplectic
manifold $M$ and an contact hypersurface $V$, with an adjusted almost complex
structure $J$, as in the previous section. For $n \in \Z_{>0}$, let $M^n =
M_-\cup V\times[-n,n] \cup M_+$ be the result of inserting a neck of length 
$2n$. 

\begin{theorem}[\cite{EGHsft},\cite{BEHWZsft}] \label{thmSFT} 
  
  Consider $L \subset M_+$ a Lagrangian submanifold and $u_n: (\D, \del \D)
  \longrightarrow (M^n, L)$ a sequence of stable $J^n$-holomorphic discs in the
  same relative homotopy class (choose $0$ as an interior marked point and $1
  \in \del \D$ a boundary marked point). Then there exists $k \in \Z_{\geq 1}$, 
 % and a sequence of reparametrizations $\phi_n: (\D, \del \D) \longrightarrow
 % (\D, \del \D)$,
 such that a subsequence of $u_n$ converges to a
  \emph{stable curve of height $k$}, also known as a \emph{holomorphic building
  of height $\underset{1}{\overset{k-2}{\vee}}$} (height 1 if $k = 2$ or $1$).

\end{theorem}

\begin{figure}[h!]
  \begin{center}
\centerline{\includegraphics[scale=0.7]{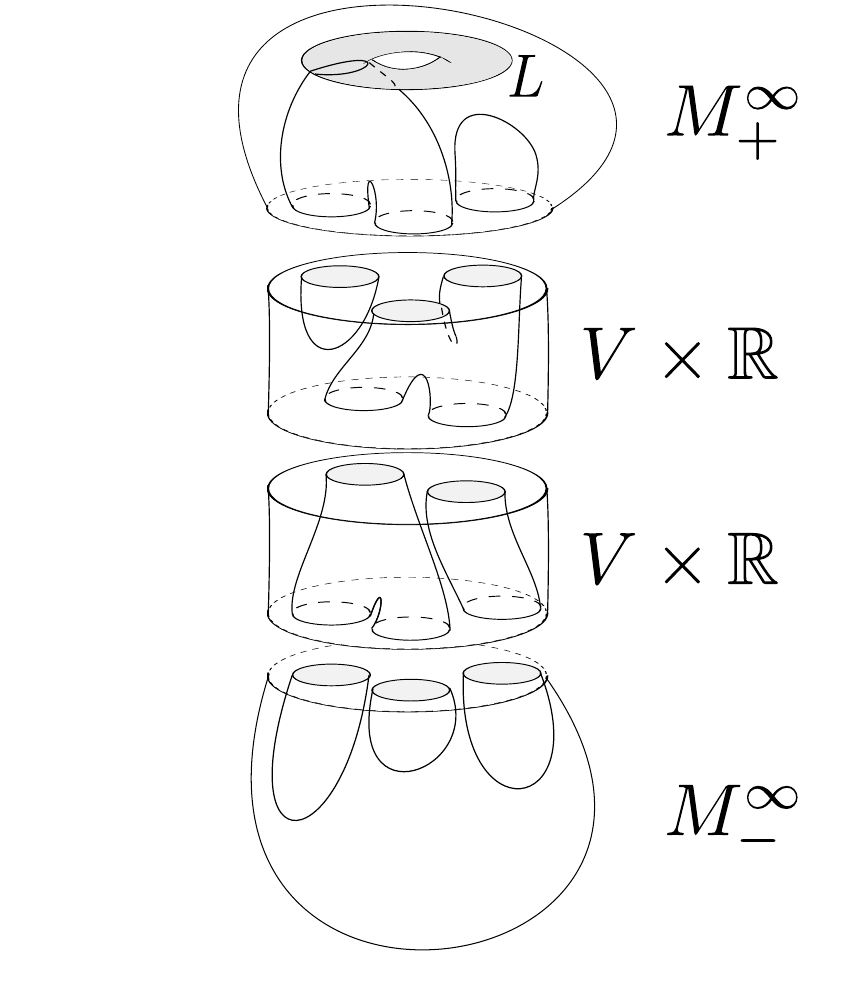}} 
\end{center}
 \caption{A stable curve of height 2 + 2.
  A possible limit of a sequence of  $J^n$-holomorphic discs.}
 \label{figBuildBloc}
\end{figure} 

For a precise definition of \emph{stable curve of height $k$} (\emph{holomorphic
building of height $\underset{1}{\overset{k-2}{\vee}}$}) we refer the reader to
section 1.6 of \cite{EGHsft} (section 9 of \cite{BEHWZsft}) - see also section
4.1 of \cite{WW}. For the notion of convergence, we refer the reader to the end of
section 9.1 of \cite{BEHWZsft}. 

Figure \ref{figBuildBloc} illustrates a typical stable curve
of height 4. It basically consists of a set of $J^\infty$-holomorphic maps from
punctured, possibly disconnected, Riemann surfaces $\Sigma_1, \dots \Sigma_k$ to
$B_1 = M_-^\infty, B_2 = V\times \R, \dots, B_{k-1} = V\times \R, B_k =
M_+^\infty$, that are asymptotic to Reeb orbits at the punctures. We label a
puncture positive/negative if it is asymptotic to a positive/negative end of
$B_i$. A negative puncture of $\Sigma_i$ is associated with a positive puncture
of $\Sigma_{i+1}$, both asymptotic to the same Reeb orbit under the respective
maps. Also, $J^\infty$ is defined in $B_i = V\times \R$, $1 < i < k$, using
translation invariance.

\section{Proof of Theorem \ref{mainthm}} \label{proof}

Before starting the setup for the proof of Theorem \ref{mainthm}, we want to
state some important preliminary results.

\subsection{Preliminary results}

Recall that we want to distinguish the tori by studying Maslov index 2 discs
they bound and applying Theorem 6.4 of \cite{RV}, which follows from the work of
Gromov \cite{Grom} - see also Proposition 4.1 A of \cite{EP}. 

\begin{lemma} (Theorem 6.4 of \cite{RV}) \label{lem_Thm6.4ofRV}
   
Let $L_0$ and $L_1$ be symplectomorphic monotone Lagrangian submanifolds of a
 symplectic manifold $(X, \omega)$, with an almost complex structure $J$ so that
 $(L_0, J)$ and $(L_1,J)$ are regular. Denote by $\varphi: X \rightarrow X$ be a
 symplectomorphism with $\varphi(L_0) = L_1$. Then the algebraic counts of
 Maslov index 2 $J$-holomorphic discs in the classes $\beta \in \pi_2(X,L_0)$ and
 $\varphi_* \beta \in \pi_2(X,L_1)$ are the same. 

\end{lemma}

In particular, we can arrange for new invariants of a monotone Lagrangian submanifold.

\begin{definition}\label{def_BdryConvexHull} 
  
   Let $L$ be a Lagrangian submanifold of a symplectic manifold $X$, endowed with
  an almost complex structure $J$. The \emph{boundary Maslov-2 convex hull} of a
  $L$ is the convex hull in $\pi_1(L)$ generated by the set \{$\del \beta \in
\pi_1(L) \ | \  \beta \in \pi_2(X,L)$, such that the algebraic count of Maslov index 2
$J$-holomorphic discs in the class $\beta$ is non-zero \}.
  
  \end{definition}

As an immediate consequence of Lemma \ref{lem_Thm6.4ofRV} and the commutative diagram

\begin{diagram}
\pi_2(X,L_0) &\rTo^{\varphi_*} & \pi_2(X,L_1)\\
\dTo_{\del} & &\dTo_{\del}\\
\pi_1(L_0) &\rTo^{\varphi_*} & \pi_1(L_1)
\end{diagram} 

we have the following Corollaries.

\begin{corollary} \label{cor_ConvexHull}

Using the same notation as in Lemma \ref{lem_Thm6.4ofRV}, for $L_0$ and
$L_1$ symplectomorphic monotone Lagrangian submanifolds of $X$, the map
$\varphi_*: \pi_1(L_0) \rightarrow \pi_1(L_1)$, sends the boundary Maslov-2
convex hull of $L_0$ to the boundary Maslov-2 convex hull of $L_1$.
   
\end{corollary}

\begin{remark} \label{rem_SLmZ}

In particular, if we are given a basis for $\pi_1(L_0)$ and a basis for
$\pi_1(L_1)$, we can see both $\pi_1(L_i)$'s as the standard lattice $\Z^m \subset \R^m$, for
some $m$. Call $\mho_{L_i}$ the image of the boundary Maslov-2 convex hull of
$L_i$, $i = 0,1$. Then Corollary \ref{cor_ConvexHull} says that $\mho_{L_0} = A
\mho_{L_1}$ for some $A \in SL(m,\Z)$.
   
\end{remark}

\begin{remark} \label{rem_NewtonPolSupPotential}
  
In the case of $L$ a Lagrangian in $X$ simply connected symplectic manifold, we
can choose an isomorphism $\pi_2(X,L) \cong \pi_2(X) \oplus \pi_1(L)$ with
classes $\alpha_i \in \pi_2(X,L)$, $i =1,\dots,m$ mapping to a basis for
$\pi_1(L)$. We can then identify $\pi_1(L)$ with the standard lattice $\Z^m
\subset \R^m$. By writing the superpotential of $L$ (see Definition 3.3 of
\cite{DA07} or Definition 2.1 of \cite{RV}) with respect to coordinates
$z_{\alpha_i}$ corresponding to $\alpha_i$ (in the sense of Lemma 2.7 of
\cite{DA07} or equation (2-2) in \cite{RV}), we can identify the boundary
Maslov-2 convex hull of $L$ with the Newton polytope of the superpotential of
$L$.

\end{remark}

We will also make use of a result given by Cho-Poddar in \cite{ChoPod} that
classifies Maslov index 2 of holomorphic discs lying on the smooth part of a toric
orbifold:

\begin{lemma} [Corollary 6.4 of \cite{ChoPod}] \label{lem_ChoPod_Cor}
  
  Let $X$ be a toric orbifold of complex dimension $n$, endowed with its
  standard complex structure, and with moment polytope whose facets are
  orthogonal to the primitive vectors $\{\bv_1, \cdots, \bv_m\}$. Then, up to the
  $T^n$-action, the \emph{smooth} Maslov index 2 holomorphic discs (not passing
  through an orbifold point) are in one-to-one correspondence with the vectors
  $\{ \bv_1, \cdots, \bv_m \}$. 

\end{lemma}

\subsection{The setup}

Consider the standard toric fibration of the orbifold $\CP\abc$, for $(a,b,c)$ a
Markov triple, represented by the left picture of Figure \ref{figABC}. We
proceed as in Corollary \ref{CorTransCut}. Perform a rational blowdown - see
Figure \ref{figRatBD} - on a neighbourhood of each vertex bounded by lens spaces
of the form $L(a^2,al_1 -1), L(b^2, bl_2 -1), L(c^2,cl_3 - 1)$ - see Proposition
\ref{mili}. Also, assume that the neighbourhoods are the quotient of balls of
some small radius in the standard coordinate chart centered in the respective
vertex (of the form $\CX(k^2,k l -1)$ for $k = a ,b ,c$), as in the paragraph
before the Corollary \ref{corExamp}. That way, we obtain an almost toric
fibration of $\CP^2$ containing $T\abc$ as a monotone fiber, as depicted in the
left base diagram of Figure \ref{figProofABC}.

\begin{figure}[h!]
  \begin{center}
\centerline{\includegraphics[scale=0.85]{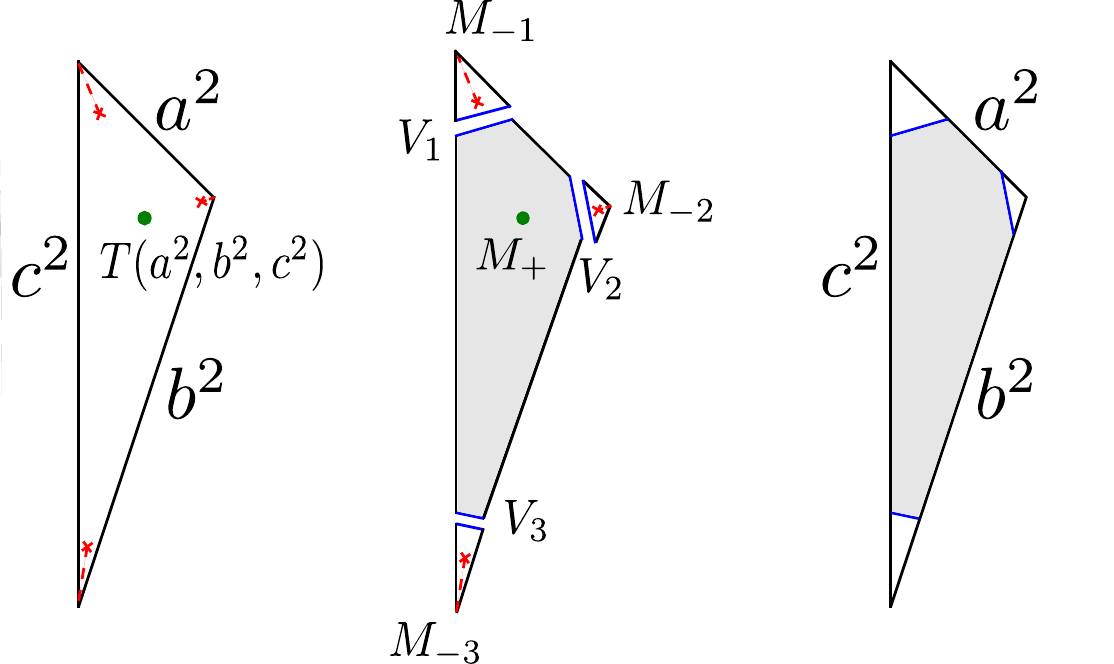}} 
 \caption{On the left, an almost toric fibration with $T\abc$ as the central fiber. 
 The middle diagram shows how the contact hypersurface $V = V_1 \cup V_2 \cup V_3$
 separates $\CP^2$ in $M_+$ and $M_- = M_{-1} \cup M_{-2} \cup M_{-3}$. The right most 
 diagram shows $M_+$ symplectically embedded into $\CP\abc$.}
\label{figProofABC}
\end{center}
\end{figure} 

Take $V = V_1 \cup V_2 \cup V_3$ the disconnected hypersurface in $\CP^2$ given
by the union of the three lens spaces $V_1 = L(a^2,al_1 -1), V_2 = L(b^2, bl_2
-1), V_3 = L(c^2,cl_3 - 1)$, used for the symplectic rational blowdown. So $V$
divides $\CP^2$ in four connected components, which we name $M_+ \cup M_{-1} \cup
M_{-2} \cup M_{-3}$, as in the middle diagram of Figure \ref{figProofABC}.
Consider the symplectic embedding of $M_+$ %$\cup \mathcal{N}$
into $\CP\abc$. %, for$\mathcal{N}$ a small collar neighbourhood of $V$. 
We take the contact structure in $V$ and the almost complex structure $J$ on 
$M_+$ %$ \cup \mathcal{N}$
coming from the standard ones of $\CP\abc$. They are adjusted in the sense given in
section \ref{ACS} (we need to take $X$ a multiple of the Liouville vector field)
- see example \ref{Examp}. 
%We take $\mathcal{N} \cong
%V_1\times(-\epsilon,\epsilon) \cup V_2\times(-\epsilon,\epsilon)\cup
%V_3\times(-\epsilon,\epsilon)$, in the way that $J$ is invariant under translation.
Since the space of compatible almost complex structures is
contractible, we have no obstruction to extend $J$ defined on $M_+$ % \cup \mathcal{N}$
to $\CP^2$. Hence, we are in shape for applying neck-stretching for
$M = \CP^2$, $V$ and $M_+$ as above, $M_- = M_{-1} \cup M_{-2} \cup M_{-3}$,
and $J$, such that, restricted to $M_+$, is given by the pullback of the 
standard complex structure in $\CP\abc$ via the embedding $M_+ \hookrightarrow \CP\abc$.

\begin{lemma} \label{lemM_+^infty}
 
 For $M$, $V$ and $J$ as above, we have that $(M_+^\infty, \omega_+^\infty,
 J_+^\infty)$ is a K\"ahler manifold isomorphic to $(\CP\abc \setminus \{p_1,
 p_2, p_3\},$ $ \omega_{std}, i)$, where $p_1, p_2, p_3$ are the pre-images of
 the vertices of the moment polytope under the standard moment map. 
 
 \end{lemma}

\begin{proof}
  Follows from Corollary \ref{corExamp}.
\end{proof}

\subsection{The proof}

We proceed to the proof of

\begin{theorem*}[\ref{mainthm}]

If $(a,b,c)$ and $(d,e,f)$ are two distinct Markov triples then the monotone
Lagrangian tori $T(a^2, b^2, c^2)$ and $T(d^2, e^2, f^2)$ are not Hamiltonian
isotopic.

\end{theorem*}

\begin{proof}

We recall that, for a monotone torus, Maslov index 2 discs have the same area.
For convenience, we normalize the symplectic forms of $M = \CP^2$ and $\CP\abc$,
so that the symplectic area of Maslov index 2 discs is 1. 

\begin{figure}[h!]
  \begin{center}
\centerline{\includegraphics[scale=1.2]{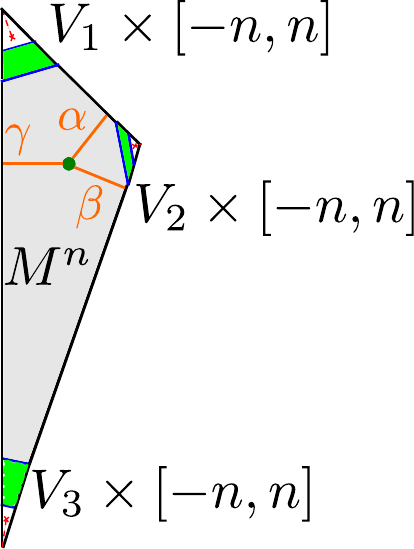}} 
 \caption{Picture of an almost toric diagram for $M^n$, which is $\CP^2$ after
we inserted three necks of length $2n$ along $V_1, V_2, V_3$. The dot in the
middle of the diagram represents $T\abc$. The segments meeting at the dot are
the images of pseudo-holomorphic discs in $M^n$ bounded by $T\abc$.}
\label{figDiscsM+}
\end{center}
\end{figure} 

Now we note that, by our embedding of $M_+$ (symplectic and holomorphic) into
$\CP\abc$, the image of $T\abc$ is the standard monotone torus in
$\CP\abc$. Hence it bounds 3 one-parameter family of Maslov index 2 holomorphic
discs, coming from the ones in $\CP\abc$ that do not pass through an orbifold
point (assuming we took a small enough neighbourhood for the rational blowdown)
- see Figure \ref{figDiscsM+} and Lemma \ref{lem_ChoPod_Cor}. By Proposition
8.3 of \cite{ChoPod}, all discs of the above mentioned families are regular. We
choose one holomorphic disc in each of the above one-parameter families. We
write $\alpha, \beta, \gamma$ for the pullbacks of these three discs to $M_+
\subset \CP^2$. 

Insert necks of length $2n$ along $V_1$, $V_2$, $V_3$ as in described in section
\ref{split}, obtaining $(M^n, e^{-n}\omega^n, J^n)$.

%\begin{lemma} \label{lem_alphabetagamma} 

%If a Maslov index 2 J-holomorphic disc $u: (\D, \del \D) \longrightarrow (M,
%T\abc)$ has its image contained in $M_+$, then it is a disc in one of the
%one-parameter families that contains either $\alpha$, $\beta$ or $\gamma$.
%\end{lemma}

%\begin{proof} Since $M_+ \hookrightarrow M^n$, we can see $u = u^\infty_+: (\D, \del
%\D) \longrightarrow (M^\infty_+, T\abc)$ is a $J^\infty_+$-holomorphic disc. From
%Lemma \ref{lemM_+^infty}, we can see $u$ as an holomorphic disc in $\CP\abc$ lying
%in the smooth part. The result follows from Corollary 6.4 of \cite{ChoPod}.    
  
%\end{proof}

Consider symplectomorphisms $\phi^n: M \longrightarrow M^n$, leaving $M_{+}
\setminus \mathcal{N}$ invariant (recall that we identify $M_{+} \subset M$ with
$M_{+} \subset M^n$), where $\mathcal{N}$ is a small neighbourhood of $V$ of the
form $\mathcal{N} \cong V_1\times(-\epsilon,\epsilon) \cup
V_2\times(-\epsilon,\epsilon)\cup V_3\times(-\epsilon,\epsilon)$. We also refer
to $M_- \subset M^n$, but note that the symplectic form defined on $M_- \subset
M^n$ differs by a factor of $e^{-2n}$ from the one defined on $M_- \subset M$ -
see the definition of $M^\tau$ in section \ref{split}. Assume also that
$\mathcal{N}$ is disjoint from $\alpha \cup \beta \cup \gamma$ (and the other
discs that came from the embedding $M_{+}\hookrightarrow \CP\abc$). 

The following Lemma (for $n=0$) says that if a $J$-holomorphic disc is not in the 
classes $[\alpha]$, $[\beta]$ or $[\gamma]$, then its image must go through 
$V = V_1 \cup V_2 \cup V_3$.

\begin{lemma} \label{lem_alphabetagamma}

Suppose that $u^n: (\D, \del \D) \longrightarrow (M^n, T\abc)$ is a Maslov index
2 J-holomorphic disc whose image does not intersect $M_- \setminus \mathcal{N}
\subset M^n$. Then it is a disc in one of the one-parameter families that
contains either $\alpha$, $\beta$ or $\gamma$.
   
\end{lemma}

\begin{proof}
   
 Consider the contact hypersurface $V \times \{-\epsilon\}$ (which is the
 boundary of $M_- \setminus \mathcal{N} \subset M^n$) inside
 $(M^n,e^{-n}\omega^n,J^n)$ and stretch the neck. We avoid developing all the
 notation for this particular neck-stretching, but let's call the `$+$'
 component of the neck stretching in the limit by $(M_+^{n + \infty},
 \omega_+^{n + \infty}, J_+^{n + \infty})$. Since the image of $u^n$ does not
 intersect $M_- \setminus \mathcal{N} \subset M^n$, we can identify $u^n$ with
 the limit $u^{n + \infty}_+: (\D, \del \D) \longrightarrow (M^\infty_+,
 T\abc)$, a $J^{n+\infty}_+$-holomorphic disc. As in Lemma \ref{lemM_+^infty},
 we are able to conclude that $(M_+^{n + \infty}, \omega_+^{n + \infty}, J_+^{n
 + \infty}) \cong (M_+^\infty, \omega_+^\infty, J_+^\infty) \cong (\CP\abc
 \setminus \{p_1, p_2, p_3\}$. Therefore, we can identify $u^n$ with an
 holomorphic disc in $\CP\abc$ lying in the smooth part. The result follows from
 Lemma \ref{lem_ChoPod_Cor}.
   
\end{proof}

By Lemma \ref{lem_alphabetagamma} for $n=0$, the discs contained in the
complement of $M_- \subset M$ are all regular. So, in order to obtain
transversality we only need to perturb $J$ on $M_-$. Therefore, we can take a
regular $J$, still having the property that, restricted to $M_+$, $J$ is given
by the pullback of the standard complex structure in $\CP\abc$, via the
embedding $M_+ \hookrightarrow \CP\abc$.

%By Lemma \ref{lem_alphabetagamma}, the discs contained in the complement of
%$M_-$ are all regular, since discs on the one-parameter families associated to
%$\alpha$, $\beta$ and $\gamma$ are regular. So, to make sure that
%$(\phi^n)^*J^n$ is regular, we only need to disturb $J$ on $M_- $. Hence, the
%set $\mathcal{J}^n_{reg}$ of almost complex structures $J$, such that,
%$(\phi^n)^*J^n$ is regular and $J$ restricted to $M_+$ is given by the pullback
%of the standard complex structure in $\CP\abc$, via the embedding $M_+
%\hookrightarrow \CP\abc$, is of second category in the Baire sense. Therefore,
%$\mathcal{J}_{reg} = \underset{n}{\bigcap}\mathcal{J}^n_{reg}$ is also of second
%category. We take $J \in \mathcal{J}_{reg}$.
Now consider a $J$-holomorphic disc $u$ of symplectic area 1. Set $[u] =
[\text{Im}(u)] \in \pi_2(\CP^2,T\abc)$, and $[\del u] \in \pi_1(T\abc)$, the class of the
boundary of $u$.

Take the isomorphism 

\begin{equation}\label{eqHomotIso}
\psi = \pi_2(\CP^2, T\abc) \overset{\cong}{\longrightarrow} \pi_2(\CP_2) \oplus
\pi_1(T\abc)
\end{equation}
  
that maps $\beta \to (0,\del \beta)$ and $\gamma \to (0 ,\del \gamma)$. Via this
identification, we completely determine the class of a disc of symplectic area 1
by its projection to the second factor, i.e., by its boundary. 
 
\begin{lemma} \label{lemma_ConvexHull}
  
Assume that the algebraic count of $J$-holomorphic discs in $[u]$ is non-zero.
Then, the class $[\del u]$ lies in the convex hull generated by $[\del \alpha]$,
$[\del \beta]$, $[\del \gamma]$ in $\pi_1(T\abc) \cong \Z^2$.
  
\end{lemma}

\begin{proof}
      
      We first observe that there is a $(\phi^n)^*J^n$-holomorphic disc
      $\tilde{u}^n$ representing $[u]$: if $(\phi^n)^*J^n$ is regular then,
      since the algebraic count of discs is nonzero, this follows from Lemma
      \ref{lem_Thm6.4ofRV}; if it is not regular then there must be a disc or
      else the condition of regularity is vacuously true. Therefore, $u^n =
      \phi^n\circ\tilde{u}^n: (\D, \del \D) \longrightarrow (M^n, T\abc)$ is
      $J^n$-holomorphic. 
 
 By Theorem \ref{CompThm}, there exists a subsequence that converges to a stable
 curve of height $k$, for some $k \geq 1$. In particular, it gives a
 $J_+^\infty$-holomorphic map $\uu : \Sigma \longrightarrow M_+^\infty$, where
 $\Sigma$ is a (possibly disconnected) punctured Riemann surface with boundary
 that consists of a circle mapped by $\uu$ to the limit of $T\abc$, which we
 call $L\abc$. One component of $\Sigma$ is a punctured disc, while the others,
 if any exist, are punctured spheres, because they cannot have positive genus. 
 
 By Lemma \ref{lemM_+^infty}, we have that
 $(M_+^\infty, \omega_+^\infty,$ $ J_+^\infty)$ is a K\"ahler manifold
 isomorphic to $(\CP\abc \setminus \{p_1, p_2, p_3\},$ $ \omega_{std}, i)$,
 where $p_1, p_2, p_3$ are the pre-images of the vertices of the moment polytope
 under the standard moment map.
  
 Hence, we can compactify $(M_+^\infty, \omega_+^\infty, J_+^\infty)$ to
 $(\CP\abc, \omega_{std}, i)$. We also extend $\uu$ to the (possibly
 disconnected) Riemann surface $\bar{\Sigma}$ as an holomorphic map in the sense
 of Definition 2.1.3 of \cite{CR} - see also Definition A of \cite{WC}.
 Topologically, we can see this map defining a class on $\pi_2(\CP\abc,L\abc)$,
 which we call $[\uu]$, since all the components of $\bar{\Sigma}$ that are not
 the disc are spheres (topologically, we could think the domain of the
 compactification of $\uu$ consists of chains of spheres attached to one disc).
 Moreover, under the identification of $T\abc \subset \CP^2$ with $L\abc \subset 
 \CP\abc$, we have that $\del[\uu] = \del[u]$. 
 
 \begin{remark} \label{rmkalpbetgamInv}
 
  Note that the discs $\alpha, \beta, \gamma$ and the symplectic form in $M_+
\subset M^n$, for all $n = 0, 1, 2, \dots, \infty$, remain invariant. We keep
calling $\alpha, \beta, \gamma$ their own limit in $\CP\abc$.  
   
 \end{remark}

% \begin{claim} \label{symparea1}
%   The symplectic area of $\uu$ is 1.
% \end{claim}
 
% \begin{proof}
  
%Take $\sigma \in \pi_2(M, T\abc)$, so that $[\del \beta]$ and $\del \sigma$ are
%generators of $\pi_1 (T\abc)$. Consider $H = [\CP^1] \in \pi_2(\CP^2)
%\hookrightarrow \pi_2(M^n, T\abc)$. We have that $[\beta], \phi^n_*\sigma$ and
%$H$ generate $\pi_2(M^n, T\abc)$. Moreover, some multiple of $H$ and some
%multiple of $\phi^n_*\sigma$ is given as a linear combination of $[\alpha],
%[\beta], [\gamma]$. Hence, by Remark \ref{rmkalpbetgamInv}, we see that the
%Maslov index remain proportional to the symplectic area, so $T\abc \subset M^n$
%(including $L\abc \subset \CP\abc$) remain monotone, with the same monotonicity
%constant. 

%Now, remember we made the splitting using $(M^n, e^{-n}\omega^n, J^n)$, so that
%in the limit the symplectic form converges to $\omega_+^\infty$ in $M_+^\infty$
%and to zero in $M_-^\infty$. This way the symplectic area of the limit (of a
%subsequence) of $u^n$ is totally concentrated in $M_+^\infty$, hence
%$[\omega_+^\infty] \cdot [\uu] = 1$.    
 
 %\end{proof}

%Using an identification $\pi_2(\CP\abc, L\abc) \overset{\cong}{\longrightarrow}
%\pi_2(\CP\abc)$ $\oplus \pi_1(L\abc)$, that maps $\beta \to (0,\del \beta)$ and
%$\gamma \to (0 ,\del \gamma)$, analogous to $\psi$ \eqref{eqHomotIso}, we can
%completely determine the relative class of a disc of area 1 by its boundary.

Call $D_1$, $D_2$, $D_3$ the inverse images of the (closed) edges of the moment
polytope of $\CP\abc$, which are (pseudo)-holomorphic curves in the sense of
Definition A of \cite{WC} (essentially the same of Definition 2.1.3 of
\cite{CR}). Also, by Definition A of \cite{WC}, the image of $\uu$ is a
(pseudo)-holomorphic curve. We will abuse notation and say that $\uu$ is a 
holomorphic curve in $\CP\abc$.

In Definition B of \cite{WC}, Chen gives a notion of algebraic intersection
number (which is a rational number) for pseudo-holomorphic curves in an
orbifold. We note that, up to relabeling, the intersection number of $\alpha,
\beta, \gamma$ with $D_1$, $D_2$, $D_3$ is $(1,0,0), (0,1,0)$ and $(0,0,1)$,
respectively. Since $\alpha$, $\beta$ and $\gamma$ generate $\pi_2(\CP\abc,
L\abc)$, the classes that have positive intersection with $D_1$, $D_2$, $D_3$
lie in the cone generated by $\alpha$, $\beta$ and $\gamma$. Moreover, this cone
projects via $\del: \pi_2(\CP\abc, L\abc) \rightarrow \pi_1(L\abc)$, to the
convex hull generated by $\del \alpha$, $\del \beta$, $\del \gamma$. To prove
Lemma \ref{lemma_ConvexHull}, it therefore suffices to prove that $\uu$
intersects $D_1$, $D_2$, $D_3$ positively.

There is an intersection formula given by Chen on Theorem 3.2 of \cite{WC}. In
particular, this formula implies that the algebraic intersection number of two
pseudo-holomorphic curves in an orbifold is positive. Since $D_1$, $D_2$, $D_3$
and $\uu$ are holomorphic, their algebraic intersection number is positive. This
implies Lemma \ref{lemma_ConvexHull}, recalling that $[\del \uu] = [\del u]$
under the identification $T\abc \cong L\abc$.
  
\end{proof}

\begin{lemma} \label{lem_algcount_alpbetgam}
  
The algebraic count of discs in the classes $[\alpha]$, $[\beta]$ and $[\gamma]$
in $\pi_2(M,T\abc)$ is $\pm 1$.
  
\end{lemma}

\begin{proof}
   
 Suppose that the algebraic count of discs in the class $[\alpha]$ is not $\pm
 1$. Because the family of $\alpha$ counts with $\pm 1$ in all $M^n$ and the
 discs on that family are regular with respect to $J^n$, by Lemma
 \ref{lem_alphabetagamma}, there must be $J^n$-holomorphic discs $u^n$ in the
 class $[\alpha] \in \pi_2(M^n,T\abc)$ intersecting $M_- \subset M^n$, for all
 $n$ (either to make $J^n$ irregular or by Lemma \ref{lem_Thm6.4ofRV}). Taking a
 limit of a subsequence, and compactifying $M_+^\infty$ to $\CP\abc$ as in the
 proof of Lemma \ref{lemma_ConvexHull}, we end up with a disc in the class
 $[\alpha] \in \pi_2(\CP\abc, L\abc)$ intersecting positively two of the complex
 curves $D_1$, $D_2$, $D_3$ (inverse images of the closed edges of the moment
 polytope of $\CP\abc$). Contradiction. Clearly, the same argument works for
 classes $[\beta]$ and $[\gamma]$.
 
\end{proof}

We have the following as an immediate corollary of Lemma \ref{lemma_ConvexHull} and Lemma 
\ref{lem_algcount_alpbetgam}:

\begin{corollary}

The \emph{boundary Maslov 2 convex hull} of $T\abc$ (Definition
\ref{def_BdryConvexHull}) is generated by $[\del \alpha]$, $[\del \beta]$ and
$[\del \gamma]$. 
   
\end{corollary}

We can see that the boundary Maslov 2 convex hull of $T\abc$ can be identified
with the polytope dual to the moment polytope of $\CP\abc$. The reader can check
that the affine lengths of the edges of the convex hull associated to $T\abc$
are $a$, $b$ and $c$. Then the Theorem \ref{mainthm} follows from Corollary
\ref{cor_ConvexHull} (which says that the boundary Maslov 2 convex hull is an
invariant among monotone Lagrangians submanifolds). More precisely, first we
choose our favourite basis for $\pi_1(T\abc)$ and $\pi_1(T(d^2,e^2,f^2)$ and see
their respective boundary Maslov 2 convex hull inside $\Z^2$. After checking
that the affine lengths of the edges are, respectively, $\{a,b,c\}$ and
$\{d,e,f\}$, we conclude that, if $\{a,b,c\} \ne \{d,e,f\}$, then $T\abc$ and
$T(d^2,e^2,f^2)$ are not symplectomorphic. This follows from Remark
\ref{rem_SLmZ}, because their boundary Maslov 2 convex hull inside $\Z^2$ are
not related via an $SL(2,\Z)$ action (note that $SL(2,\Z)$ preserves affine
length).

\end{proof}

\end{document}